\documentclass[paper=a4,headings=normal,bibliography=toc,index=totoc,
	]{scrartcl}
	
\usepackage[utf8]{inputenc}
\usepackage{lmodern}
\usepackage[T1]{fontenc}
\usepackage[english]{babel}

\ifpdfoutput{\usepackage{microtype}}{}

\usepackage{fixltx2e}  

\usepackage{amsmath,amsfonts,amssymb,amsthm}

\usepackage{textcomp}
\usepackage{exscale} 

\usepackage{cancel}   
\usepackage{graphicx} 
\usepackage{url} 
\usepackage[arrow, matrix, curve]{xy}
\usepackage{xspace}
 
\usepackage{soul}

\usepackage[makeindex]{splitidx}
\newindex[Index of notation]{not}
\newindex[Index]{idx}

\usepackage[usenames,dvipsnames]{xcolor}

\ifpdfoutput{
\usepackage[pdftex,hypertexnames=false,plainpages=false,pdfpagelabels=true,setpagesize=false]{hyperref}\hypersetup{breaklinks=true,linktocpage=false,colorlinks=true,urlcolor=black,linkcolor=black,citecolor=black,
}}{}

\newcommand{\bel}[1]{\begin{equation}\label{#1}}

\newcommand{\be}{\begin{equation}}
\newcommand{\ba}{\begin{eqnarray}}
\newcommand{\ea}{\end{eqnarray}}
\newcommand{\rf}[1]{(\ref{#1})}
\newcommand{\bi}{\bibitem}
\newcommand{\qe}{\end{equation}}

\newcommand{\R}{\mathbb{R}}

\newcommand{\N}{\mathbb{N}}
\newcommand{\Np}{\mathbb{N}_+}

\newcommand{\D}{\Delta}

\newcommand{\map}{\longrightarrow}
\newcommand{\too}{\longrightarrow}

\newcommand{\half}{\frac12}
\newcommand{\intlim}{\int\limits}

\newcommand{\smin}{\setminus}

\providecommand{\ce}{\mathrel{\mathop:}=}

\providecommand{\co}{\colon}
\providecommand{\itind}[1]{\textit{#1}\index{#1}}
\providecommand{\txind}[1]{#1\index{#1}}

\providecommand{\cl}[1]{\overline{#1}}

\providecommand{\com}[1]{}

\providecommand{\abs}[1]{\lvert#1\rvert}

\providecommand{\set}[1]{\lbrace#1\rbrace}

\providecommand{\bigset}[1]{\big\lbrace#1\big\rbrace}
\providecommand{\Bigset}[1]{\Big\lbrace#1\Big\rbrace}

\DeclareMathOperator{\Div}{div}
\DeclareMathOperator{\supp}{supp}

\newcommand{\bd}{\partial}
\newcommand{\dd}[2]{\frac{\partial#1}{\partial#2}}

\newcommand{\ddd}[3]{\frac{\partial^2#1}{\partial{#2}\partial{#3}}}

\providecommand{\dcd}{\,\cdot\,{,}\,\cdot\,}
\providecommand{\fdt}{\,\cdot\,}

\newcommand{\leb}{\lambda\hspace{-5 pt}\lambda}

\renewcommand{\phi}{\varphi}

\newenvironment{eqn*}{\begin{equation*}}{\end{equation*}}

\newcommand{\ie}{i.\,e.\@\xspace}
\newcommand{\eg}{e.\,g.\@\xspace}
\newcommand{\cf}{cf.\@\xspace}

\newcommand{\wrt}{w.\,r.\,t.\@\xspace}

\newcommand{\resp}{resp.\@\xspace}
\newcommand{\zeroth}{0{th}\xspace}
\newcommand{\first}{1{st}\xspace}

\newcommand{\ord}{-{th}\xspace}

\providecommand{\WF}{Wright--Fisher\xspace}

\providecommand{\KBE}{Kolmogorov backward equation\xspace}
\providecommand{\KFE}{Kolmogorov forward equation\xspace}

\setkomafont{captionlabel}{\small\itshape}

\numberwithin{equation}{section} 
\swapnumbers
 
\theoremstyle{plain} 
\newtheorem{thm}{Theorem}[section]
\newtheorem{prop}[thm]{Proposition}
\newtheorem{lem}[thm]{Lemma}
\newtheorem{cor}[thm]{Corollary}

\theoremstyle{definition}
\newtheorem{dfi}[thm]{Definition}

\newtheorem{rmk}[thm]{Remark}

\allowdisplaybreaks[1]

\definecolor{gruen}{rgb}{0.6,0,0.5}

\theoremstyle{plain}
\newtheorem*{thm*}{Theorem}

\begin{document}
\title{A hierarchical extension scheme for solutions of the \WF model}

\author{Julian Hofrichter\footnote{hofricht@mis.mpg.de}, Tat Dat Tran\footnote{trandat@mis.mpg.de}, Jürgen Jost\footnote{jost@mis.mpg.de, phone +49-341-9959550}}

\date{\today}

\maketitle

\begin{abstract}
We develop a global and hierarchical scheme for the forward  Kolmogorov (Fokker-Planck) equation of the diffusion approximation of the Wright-Fisher model of population genetics. That model describes the random genetic drift of several alleles at the same locus in a population. The key of our scheme is to connect the solutions before and after the loss of an allele. Whereas in an approach via stochastic processes or partial differential equations, such a loss of an allele leads to a boundary singularity, from a biological or geometric perspective, this is a natural process that can be analyzed in detail. Our method depends on evolution equations for the moments of the process and a careful analysis of the boundary flux. 
\end{abstract}


\section{Introduction}
The Wright-Fisher model \cite{fisher,wright1}, the basic model of population genetics, models  random genetic drift in a finite population of fixed size. It describes the evolution of the probabilities between non-overlapping generations  of several alleles, that is, competing alternatives for representation at a genetic locus. These probabilities  are obtained from random sampling in the parental generation. In more detail,  we have a population of finite size $N$ carrying initially $n+1$ different alleles, at a single locus in the basic version that we are concerned with here. For each member of a new generation -- the model works with discrete time steps --, randomly (with replacement) a mother is drawn from the previous generation who then donates her allele to that offspring. Eventually, all but one allele will get lost by random drift from the population, because it may happen that at some time, by chance no carrier of a particular allele is chosen as a mother for a member of the next generation. Then in that and all future generations, that particular allele will no longer be represented. Therefore, almost surely, asymptotically only a single allele will survive.
 For the mathematical analysis,  starting with the pioneering work of Kimura \cite{kimura1,kimura2,kimura3},  one goes to the diffusion approximation of the process, that is, rescales time as $t=\frac{1}{N}$ and lets $N\to \infty$. The evolution of the probability distribution of the alleles in the population is then described by the so-called  forward Kolmogorov equation
 \begin{align}\label{eq_Ln_pre}
\dd{}{t} u(x,t) = \half\sum_{i,j=1}^n\ddd{}{x^i}{x^j}\big(x^i(\delta^i_j-x^j)u(x,t)\big)	&\text{\quad in $(\D_n)_\infty=\D_n\times(0,\infty)$,}
 \end{align}
which is also known as the Fokker-Planck equation (\cf also section~\ref{sec_kolmo} for more details). The state space here is the $n$-dimensional probability simplex $\D_n$. That is, instead of the original model of a finite population evolving in discrete time steps, one rather considers the diffusion approximation for an infinite population in continuous time.

The basic model just described covers a single genetic locus only. Extensions to several loci are possible, as is the inclusion of  mutation, selection, or a spatial population structure, and this has driven research in  mathematical population genetics  (\cite{ewens,buerger}). The forward Kolmogorov equation is  a partial differential equation of parabolic type. There is another important PDE associated with this process,  the backward Kolmogorov equation
 \begin{align}\label{eq_Ln*_pre}
\dd{}{t} u(x,t) =   \half\sum_{i,j=1}^n\big(x^i(\delta^i_j-x^j)\big)\ddd{}{x^i}{x^j}u(x,t)	&\text{\quad in $(\D_n)_\infty=\D_n\times(0,\infty)$,}
 \end{align}
which is the  adjoint of the forward equation \wrt a suitable product (and with a boundary contribution whose analysis will be the main new tool of the present paper). 

The forward equations describes how an initial probability  distribution for the alleles in a population evolves in time. Since, as explained,  alleles can disappear from the population by random genetic drift and will then be lost forever (unless the possibility of mutations is included, which, however, we do not consider in this paper), one needs to dynamically connect  strata corresponding to different allele numbers. This is the main source of difficulties addressed in this paper, particularly as the operator in equation~\eqref{eq_Ln_pre} becomes degenerate towards the boundary of the domain. 

We note that we cannot prescribe boundary values for the \KFE \rf{eq_Ln_pre}. In terms of the process, the boundary plays a passive role as there is flux from the interior into the boundary, but not in the other direction. Therefore, the boundary values are determined by the evolution of the solution in the interior. In analytical terms, we have an exit boundary in the terminology of Feller\cite{feller}. 

The backward equation, in contrast, describes the evolution of the probabilities of the distributions of ancestral states giving rise to the present allele distribution. Solutions of an inhomogeneous  backward equation also yield formulas for expected loss of allele times. In the backward equation, the transition between strata is simpler, because the contributions from the various strata essentially add up. That is, the degeneracy  at the boundary is easier for the backward than for the forward equations, and there already exists a substantial literature deriving corresponding expansions for the solution of the backward equations, see for instance  \cite{LF1975,Shi1977,Gri1979,Gri1980,Shi1981,Tav1984,EG1993}. A powerful tool in this line of research has been Kingman's coalescent \cite{King1982}, that is, the method of tracing lines of descent back into the past and analyzing their merging patterns (for a quick introduction to that theory, see also \cite{jost_bio}). We shall describe some of these achievements in more detail below.

Here, as already mentioned, we analyze the more difficult boundary transitions for the forward equation. For this purpose, we shall also utilize the duality between the forward and the backward equation in an essential manner. Nevertheless, we should emphasize that the solutions of the forward equation are different from those of the backward equation and cannot be recovered from the latter. Therefore, the results about solutions of the backward equation  that we shall reference below do not directly apply to the forward equation. Therefore, we develop a different approach in this paper. 

Since the original work of Fisher, Wright and Kimura, the Wright-Fisher model as well as several extensions or generalizations of it have been investigated in considerable detail. One research line incorporated the model into the general theory of stochastic processes or partial differential equations  and derived existence and regularity results from that general perspective, see for instance \cite{ethier1,ethier2,ethier3,ethier4,karlin} for stochastic processes or \cite{epstein1,epstein2} for partial differential equations. A lot of research was also concerned with explicit results and formulas, for instance for the expected time of loss of an allele, and therefore, while relying on the abstract theory, had to work out the specific and explicit structure of the model.  We propose a third line of research on the Wright-Fisher model, based on geometric constructions, more specifically on information geometry, that is, the geometry of probability distributions, see \cite{amari,ajls}. 

In fact, the structure of the model is surprisingly rich, in particular when viewed from a geometric perspective. We are carrying out a systematic research project (see \cite{dat,julian,THJ1,THJ2,THJ3}) to develop this structure and thereby obtain an understanding of the Wright-Fisher model that is both deeper and more explicit, in the sense that precise formulae for the quantities of interest can be derived. The present paper, which is based on \cite{julian}, is a key component of this project.  

There are many solution schemes for the  \KFE\eqref{eq_Ln_pre} in the literature. As early as 1956, Kimura presented a local solution scheme for the 3-allelic case ($n=2$) in~\cite{kimura3}. Another approach by separation of variables was presented by Baxter, Blythe and McKane in~\cite{BBMcK}, this time for an arbitrary number of alleles. For the \KBE , the situation is even better. The starting point of much of the literature was the observation of Wright \cite{wright2} that when one includes mutation, the degeneracy at the boundary is removed. And when the probability of a mutation of allele $i$ into allele $j$ depends only on the target $j$, then the backward process possesses a unique stationary distribution, at least as long as those mutation rates are positive. This then lead to explicit representation formulas in \cite{LF1975,Shi1977,Gri1979,Gri1980,Shi1981,Tav1984,EG1993}. Some of these formulas, in particular those of \cite{Shi1981,EG1993} also pertain to the limit of vanishing mutation rates. In \cite{Shi1981}, a superposition of the contributions from the various strata was achieved whereas \cite{EG1993} could write down an explicit formula in terms of a Dirichlet distribution. As already mentioned, for the \KFE, the situation is more subtle because such a superposition for the contributions from the various strata no longer holds. The difference is the following. The backward equation produces the probability distribution of ancestral states giving rise to a current distribution. That latter distribution may involve states with different numbers of alleles present. Their ancestral distributions, however, do not interfere, regardless of the numbers of alleles they involve. Thus, some superposition principle holds, and the \KBE nicely extends to the boundary. In contrast, the \KFE yields the future distribution evolving from a current one. Here, however, the probability of some boundary state does not only depend on the flow within the corresponding boundary stratum, but also on the distribution in the interior, because at any time, there is some probability that an interior state loses some allele and turns into a boundary state. Thus, there is a continuous flux into the boundary strata from the interior. Therefore, the extension of the flow from the interior to the boundary strata is different from the intrinsic flows in those strata, and no superposition principle holds. 

Let us now describe in more specific terms what we achieve in this paper. The key is the degeneracy at the boundary of the Kolmogorov equations. While from an analytical perspective, this presents a profound difficulty for obtaining boundary regularity of the solutions of the equations, from a biological or geometric perspective, this is very natural because it corresponds to the loss by random drift of some alleles  from the population in finite time. And from a stochastic perspective, this has to happen almost surely. Now, however, even after an allele gets lost, the population keeps evolving by random genetic drift according to the Wright-Fisher scheme, simply with fewer alleles than before, until only one allele is left and the evolution comes to a halt (in this basic model). Therefore, it is biologically essential and geometrically natural to connect the processes before and after the loss of an allele. When the original process, starting with, say, $n+1$ alleles, takes place on an $n$-dimensional probability simplex, after the loss of an allele, we have a process on an $(n-1)$-dimensional simplex. The latter then should be identified as a facet of the former, that is, the loss of an allele simply means that the process moves from the interior into the boundary of the original simplex. Of course, this will then be repeated when further alleles get lost, and the process moves to lower and lower dimensional boundary strata until it gets stuck in a corner. In this paper, we therefore construct a global solution that incorporates and connects these successive loss of allele events, that is
\begin{thm*}[\cf~Theorem 4.12 on p.\,\pageref{thm_moments_n}]\label{thm_moments_n_pre}
For $n\in\N$ and a given  initial condition $f\in L^2(\D_n)$, the \KFE \eqref{eq_Ln_pre} corresponding to the diffusion approximation of the  $n$-dimen\-sional \WF model  possesses a unique {extended solution} 
$U\co{\big(\overline{\Delta}_{n}\big)}_{\infty}\map\R$,
which is defined on the entire closed simplex. 
\end{thm*}
The key to our solution concept is the evolution of the moments of the underlying process. These moments are global quantities, and their evolutions include what flows into the boundary. For instance, the total mass of the process remains 1. The moment evolution equations constitute an infinite family of ODEs. We have
 $\dd{}{t}\bar{\mu}_0(t)=0$ (preservation of the total mass) and  
\begin{align}\label{eq_moments_n-intro}
\dd{}{t}\bar{\mu}_\alpha(t)=-\frac{\abs{\alpha}(\abs{\alpha}-1)}{2}\bar{\mu}_\alpha(t)+\sum_{i=1}^n\frac{\alpha_i(\alpha_i-1)}{2}\bar{\mu}_{\alpha-e_i}(t) 
\end{align}
where the moment $\bar{\mu}_\alpha$ is the expectation value of $x^\alpha$, for $\alpha=(\alpha_1,\dotsc,\alpha_n)$, $\abs{\alpha}\geq 1$ ($e_i$ denotes the multi-index $(0,\dotsc,0,1,0,\dotsc,0)$ with 1  at the $i$\ord position). Our global solution respects and reflects all these moment evolutions, and is in turn determined by them. In that sense, the moment evolution equations \rf{eq_moments_n-intro} are equivalent to (our solution of) the \KFE .

In technical terms, this concept of (the extended) solution involves developing a hierarchical  scheme that relies on equations for the moments of the process, the interplay between the forward and the backward Kolmogorov equations, representations of the solutions in terms of Gegenbauer polynomials and on a careful analysis of the boundary flux. We should also point out  that while the existing literature draws upon concepts and results from the theory of stochastic processes in an essential way, our approach is analytical and geometric and therefore offers an alternative to the existing ones.  The processes for fewer alleles that occur in our hierarchical scheme combine the intrinsic dynamics with a reduced allele number with the contributions through loss-of-allele events from larger allele sets. This avoids any singularities. 

In geometric terms, we carefully investigate the boundary flux from a simplex into its various boundary faces. In analytical terms, we reduce the Kolmogorov type PDE to a family of ODEs, the moment evolution equations. Earlier research in that direction is due to Dawson-Hochberg \cite{dawson} and Dynkin \cite{dynkin}, within a more general and therefore less explicit scheme. In our approach, the moment equations are global and guarantee the consistency of the process beyond loss-of-allele events and across the various boundary faces, see Theorem \ref{thm_moments_n}. And in \cite{THJ1,THJ2,THJ3}, we have already constructed a global solution, which (implicitly) made  use of the hierarchical scheme presented here. 

From the PDE perspective \cite{epstein1,epstein2}, it might be of interest to see how a PDE with singular behavior at a singular boundary -- recall the lower-dimensional boundary strata where the boundary has corners -- can be explicitly solved by a hierarchical scheme. 

Extensions to several loci and the effects of genetic recombination, processes with mutations, presence of selective forces, etc. will be studied elsewhere. 

\section{Preliminaries and notation}\label{sec_simplex}

Since we plan to develop a hierarchical scheme for the solution of the Kolmogorov equations on the various boundary strata of the standard simplex, we need to develop some notation  for the recursive application of our scheme on different boundary strata. We also need suitable hierarchical products. That is the purpose of this section.

We consider relative frequencies $x^0,x^1,\dots ,x^n$ of alleles $0,1,\dots,n$. Thus we have the normalization $\sum_{j=0}^n x^j=1$, and we  therefore have $x^0=1-\sum_{i=1}^n x^i$. We shall therefore work with the 
 (open) \textit{$n$-dimen\-sional standard orthogonal simplex}\sindex[not]{Dn@$\D_n$} 
\begin{align}\label{eq_stand_simpl}
\D_n\ce\Bigset{{(x^1,\dotsc,x^n)\in\R^n\big\vert x^i > 0\text{ for $i=1,\dotsc,n$ and }\sum_{i=1}^n x^i < 1}},
\end{align}
or equivalently,
\begin{align}
\D_n =\Bigset{{(x^1,\dotsc,x^n)\in\R^n\big\vert x^j > 0\text{ for }j=0,1,\dotsc,n \text{ and }\sum_{j=0}^n x^j=1}}.
\end{align}
Its topological closure is 
\begin{align}\label{eq_simpl_I_n}
\cl{\Delta}_{n}=\bigset{{(x^1,\dotsc,x^n)\in\R^n\big\vert x^i \ge 0 \text{ for }i=1,\dotsc,n \text{ and }\sum_{i=1}^n x^i \le 1}}.
\end{align}
In order to include time $t\in [0,\infty)$, we shall also use the notation
\begin{equation*}
  (\D_n)_\infty:=\D_n\times(0,\infty). 
\end{equation*}

The boundary $\bd\D_n=\cl{\D}_n\smin\D_n$ consists of various subsimplices of descending dimensions called \textit{faces}, starting from the $(n-1)$-dimen\-sional \textit{facets} down to the \textit{vertices} (which represent 0-dimen\-sional faces). Each subsimplex of dimension $k\leq n-1$ is isomorphic to the $k$-dimen\-sional standard orthogonal simplex $\D_k$. For an index set $I_k=\{i_0,i_1,\dots ,i_k\}\subset \set{0,\dotsc,n}$ with $i_j\neq i_l$ for $j\neq l$ 
we put 
\begin{align} 
\Delta_{k}^{(I_{k})}\ce\Bigset{{(x^1,\dotsc,x^n)\in{\overline{\Delta}_n}\big\vert x^i > 0\text{ for $i\in I_k$; }x^i=0\text{ for $i\in I_n\smin I_k$}}}.
\end{align}
We note that $\D_n =\Delta_{n}^{(I_{n})}$.

For a given $k\leq n-1$, there are of course $\binom{n+1}{k+1}$ different such subsets $I_k$ of $I_n$, each of which corresponds to a certain boundary face $\D_k^{(I_k)}$. We therefore introduce the \textit{$k$-dimen\-sional boundary $\bd_k\D_n$ of $\D_n$}\sindex[not]{dkDn@$\bd_k\D_n$} by putting
\begin{align}\label{eq_bd_k}
\bd_k\D_n^{(I_{n})}\ce \bigcup_{I_k\subset I_n}\Delta_k^{(I_k)}\subset \bd\D_n^{(I_{n})}\quad\text{for $0\leq k\leq n-1$}.
\end{align}
With this notation, we have  $\bd_n\D_n =\D_n$, although this is not a boundary component. The concept of the $k$-dimen\-sional boundary also applies to simplices which are themselves boundary instances of some $\D_l^{(I_l)}$, $I_l\subset I_n$ for $0\leq k < l\leq n$, thus
\begin{align}
\bd_k\D_l^{(I_l)}=\bigcup_{I_k\subset I_l}\Delta_k^{(I_k)}\subset \bd\D_l^{(I_l)}.
\end{align}

In the \WF model,  
 $\D_n$ corresponds to the state where all $n+1$ alleles are present, whereas $\bd_k\D_n$ represents the state where exactly (any) $k+1$ alleles are present in the population. An individual $\D_k^{(\set{i_0,\dotsc,i_{k}})}$ in $\bd_k\D_n$ corresponds to the state where exactly the alleles $i_0,\dotsc,i_{k}$ are present in the population. Likewise, $\bd_{k-1}\D_k^{(\set{i_0,\dotsc,i_{k}})}$ corresponds to the state where exactly one further allele out of $i_0,\dotsc,i_{k}$ is eliminated from the population.\\

In order to define integral  products on $\D_n$ and its faces, 
\sindex[not]{L2UD@$L^2\Big(\bigcup_{k=0}^n\bd_k\D_n\Big)$} we need appropriate spaces of square integrable functions. Thus, 
\begin{multline}\label{eq_dfi_L2_union}
L^2\Big(\bigcup_{k=0}^n\bd_k\D_n\Big)
\ce\Big\{f\co\cl{\D}_n\too\R\,\Big\vert\,\text{$f\vert_{\bd_k\D_n}$ is $\leb_k$-measurable and}\\\text{$\int_{\bd_k\D_n} \abs{f(x)}^2\,\leb_k(dx) < \infty$ for all $k=0,\dotsc,n$}\Big\}.
\end{multline}
Here, $\leb_k$ is the  $k$-dimensional Lebesgue measure, but when integrating over some ${\Delta_k^{(I_k)}}$ with $0\notin I_k$, the measure needs to be replaced with the one induced on ${\Delta_k^{(I_k)}}$ by the Lebesgue measure of the containing $\R^{k+1}$. That  measure, however, will still be denoted by $\leb_k$\label{pag_leb_k} as it is clear from the domain of integration ${\Delta_k^{(I_k)}}$ with either $0\in I_k$ or $0\notin I_k$ which version is actually used. In particular,  for the top-dimen\-sional simplex, we simply have 
\sindex[not]{L2Dn@$L^2(\D_n)$}
\begin{align}
L^2(\D_n)\ce\Big\{f\co\D_n\too\R\,\Big\vert\,\text{$f$ is $\leb_n$-measurable and $\int_{\D_n}\abs{f(x)}^2\,\leb_n(dx) < \infty$}\Big\}.
\end{align}
Furthermore,  we also define for $k\in\N\cup\set{\infty}$
\sindex[not]{C0@$C_0^k(\cl{\Delta}_n)$}\sindex[not]{Cc@$C_c^k(\cl{\Delta}_{n})$}
\begin{align}
C_0^k(\cl{\Delta}_{n})&\ce\bigset{f\in C^k(\cl{\Delta}_{n})\big|f\vert_{\bd\D_n}=0},\\
C_0^k({\Delta}_{n})&\ce\bigset{f\in C^k({\Delta}_{n})
\big|\exists\,\bar{f}\in C_0^k(\cl{\Delta}_{n})\text{ with }\bar{f}\vert_{\D_n}=f}\\[-0.7em]
\intertext{as well as}\notag\\[-2.7em]
C_c^k(\cl{\Delta}_{n})&\ce\bigset{f\in C^k(\cl{\Delta}_{n})\big|\supp(f)\subsetneq{\Delta}_{n}}\label{eq_def_Cc}.
\end{align}

We can now  introduce a \textit{product} of functions $u,v\in L^2(\D_n)$ by\sindex[not]{<@$(\dcd)_n$}
\begin{align}\label{eq_prod}
(u, v)_n\ce \intlim_{{\Delta}_n} u(x)v(x)\,\leb_n(dx).
\end{align}
Importantly, we integrate here  only over the interior $\D_n$; the index -- if no confusion is to be expected -- may be omitted. As~$n$ was arbitrary, the product may also be recursively applied on some ${\Delta_k^{(I_k)}}\subset\bd\D_n$.

Utilising the various products $(\dcd)_k$ on ${\bd_k\Delta_n^{(I_n)}}$ for $k=0,\dotsc,n$, we can now  define a \textit{hierarchical product} on the closure of the simplex $\overline{\Delta}_{n}$. For functions $u,v\co\overline{\Delta}_{n}\too\R$ with $u,v\vert_{{\bd_k\Delta_n^{(I_n)}}}\in L^2\big(\bd_k\Delta_n^{(I_n)}\big)$ 
for all $k=0,\dotsc,n$, we put\sindex[not]{[,]@$[\dcd]_n$}
\begin{align}\label{eq_prod_[]}
[u,v]_n\ce \sum_{k=0}^n (u,v)_k
\end{align}
with $(u,v)_k$  denoting the integral over the full $k$-dimen\-sional boundary $\bd_k\D_n$ of $\D_n$ (\cf  equation~\eqref{eq_bd_k}), thus 
\begin{align}\label{eq_prod_delta_k}
[u,v]_n=\sum_{k=0}^n (u,v)_k=\sum_{k=0}^n\, \intlim_{\bd_k\Delta_n} u(x)v(x)\,\leb_k(dx)
=\sum_{k=0}^n\, \sum_{{I_k\subset I_n}} \intlim_{\Delta_k^{(I_k)}} u(x)v(x)\,\leb_k(dx)\quad ;
\end{align}
here, $\leb_k$ again denotes either the Lebesgue measure of $\R^k$ or -- if the domain of integration is some ${\Delta_k^{(I_k)}}$ with $0\notin I_k$ -- the measure induced on ${\Delta_k^{(I_k)}}$ by the Lebesgue measure of the containing $\R^{k+1}$. 

\section{The Kolmogorov equations}\label{sec_kolmo}

The \textit{\KFE} for the diffusion approximation of the $n$-allelic 1-locus \WF model reads \sindex[not]{Dninfty@$(\D_n)_\infty$}\sindex[not]{u(x,t)@$u(x,t)$}
\begin{equation}\label{eq_Ln}
\begin{cases}
\dd{}{t} u(x,t) = L_n u(x,t)	&\text{in $(\D_n)_\infty=\D_n\times(0,\infty)$}\\
u (x,0) = f(x)			&\text{in $\D_n$, $f\in L^2(\D_n)$}\\
\end{cases}
\end{equation}
for $u(\,\cdot\,,t)\in C^2(\D_n)$ for each fixed $t\in(0,\infty)$ and $u(x,\,\cdot\,)\in C^1((0,\infty))$ for each fixed $x\in\D_n$ and with
\begin{align}\label{eq_Ln_def}
 L_n u(x,t) \ce \half\sum_{i,j=1}^n\ddd{}{x^i}{x^j}\big(x^i(\delta^i_j-x^j)u(x,t)\big)
\end{align}
being the \textit{forward operator}\sindex[not]{Ln@$L_n$}. Analogously, we have the \textit{backward operator}\sindex[not]{Ln*@$L_n^*$}
\begin{align}\label{eq_Ln*_def}
L_n^* u(x,t) \ce \half\sum_{i,j=1}^n\big(x^i(\delta^i_j-x^j)\big)\ddd{}{x^i}{x^j}u(x,t),
\end{align}
appearing in the corresponding \KBE. The definitions of the operators given in equations~\eqref{eq_Ln_def} and~\eqref{eq_Ln*_def} also apply to the closure $\cl{\D}_n$, and in fact, we shall also consider extensions of the solution and the differential equation to the boundary. However,  the operators become degenerate at the boundary. In fact, on the boundary, the corresponding entries in the coefficient matrix $(x^i(\delta^i_j-x^j))_{ij}$ vanish, thus the operators are not uniformly elliptic on $\D_n$. 

Later, we will also use a \textit{weak formulation} of the \KFE
\begin{align}\label{eq_weak_pre}
\bigg[\dd{}{t}{U(t)},\phi\bigg]_n=\big[{U(t)},L_n^*\phi\big]_n\quad \text{for $\phi\in C^\infty(\cl{\D}_n)$ and all $t\in(0,\infty)$}.
\end{align}

\subsection*{Properties and eigenfunctions}

For relations between the two operators, we immediately have the following two lemmas:
\begin{lem}\label{lem_adjoint}
$L_n$ and $L^*_n$ are (formal) adjoints with respect to the product $(\dcd)_n$ in the sense that
\begin{align}
(L_n u, \phi)_{n}= (u,L^*_n \phi)_n\quad\text{for $u\in C^2(\cl{\Delta}_n)$, $\phi\in C^2_0(\cl{\Delta}_n)$.}
\end{align}
\end{lem}
\begin{proof}
The assertion directly follows from proposition~\ref{prop_mgf} below.
\end{proof}

\begin{lem}\label{lem_ef-shift}\sindex[not]{on@$\omega_n$}
For an eigenfunction $\phi\in C^2(\cl{\Delta}_n)$ of $L_n$ and $\omega_n\ce\prod^n_{k=1}x^k\big(1-\sum_{l=1}^n x^l\big)$, we have: $\omega_n\phi\in C^2_0(\cl{\Delta}_n)$ is an eigenfunction of $L^*_n$ corresponding to the same eigenvalue and conversely.
\end{lem}
\begin{proof}
Looking for a function $\omega_n$ with $L_n^*(\omega_n u)=\omega_n L_n (u)$ (and hence for $L_n$-eigenfunctions $\phi$ consequently $L_n^*(\omega_n\phi)=\omega_n L_n (\phi)=-\lambda\omega_n\phi$),
we have on the one hand 
\begin{align}
L_n u=-\frac{n(n+1)}{2}u+\sum_{i}(1-(n+1)x^i)\dd{}{x^i}u+\half\sum_{i,j}x^i(\delta^i_j-x^j)\dd{}{x^i}\dd{}{x^j} u
\end{align}
and on the other hand
\begin{align}\notag
L_n^*(\omega_n u)
&=\half\sum_{i,j}x^i(\delta^i_j-x^j)\dd{}{x^i}\dd{}{x^j} (\omega_n u)\\
&=\half\sum_{i,j}x^i(\delta^i_j-x^j)
\Big(\Big(\dd{}{x^i}\dd{}{x^j} \omega_n\Big) u+2\dd{}{x^j}\omega_n\dd{}{x^i}u+\omega_n\dd{}{x^i}\dd{}{x^j}u\Big).
\end{align}
Thus, it suffices that  such a function $\omega_n$ satisfies
\begin{align}
\begin{cases}
\sum_{i,j}x^i(\delta^i_j-x^j)\dd{}{x^i}\dd{}{x^j} \omega_n=-{n(n+1)}\omega_n\\
\sum_{j}x^i(\delta^i_j-x^j)\dd{}{x^j}\omega_n=(1-(n+1)x^i)\omega_n\quad\text{for all~$i$},
\end{cases}
\end{align}
which is the case for $\omega_n=\prod^n_{k=1}x^k\big(1-\sum_{l=1}^n x^l\big)$ as may easily be verified by direct computation.
\end{proof}

For our hierarchical scheme,  it will be important that the operator $L_n^*$, if restricted to subsimplices $\D_k^{(I_{k})}\cong\D_k$ in $\bd{\Delta^{(I_n)}_n}$ of any dimension~$k$,  is the adjoint of the differential operator $L_k$ corresponding to the evolution of a $(k+1)$-allelic process in $\Delta_k$:

\begin{lem}\label{lem_restr_bwd}
For $0\leq k < n$ and $I_k\subset\set{0,\dotsc,n}$, $\abs{I_k}=k$, we have
\begin{align}
L_n^*\big\vert_{\D_k^{(I_{k})}}=L_k^*.
\end{align}
\end{lem}
\begin{proof}
For $I_k\subset\set{1,\dotsc,n}$, $\abs{I_k}=k$, we directly have:
\begin{align}\notag
L_n^*\big\vert_{\D_k^{(I_{k})}}
&=\half\sum_{i,j=1}^n\big(x^i(\delta^i_j-x^j)\big)\ddd{}{x^i}{x^j}\Big\vert_{\D_k^{(I_{k})}}\\
&=\half\sum_{i,j\in I_k}\big(x^{i}(\delta^{i}_{j}-x^{j})\big)\ddd{}{x^{i}}{x^{j}}
\equiv L_k^*.
\end{align}
By symmetry, this then  also holds for $I_k$ with $0\in I_k$, hence for arbitrary $I_k$. 
\end{proof}
We may therefore omit the index~$k$ in $L_k^*$ whenever convenient, in particular when considering domains where (parts of) the boundary are included. For the  operator $L_n$, in contrast, we do not have such a restriction property\label{pag_op_restr_fwd}.

The starting point of our solution scheme will be the  solution constructed in \cite{THJ2}. That solution depends on
\begin{prop}\label{prop_Gegenb_n}
For $n\in\Np$ and each multi-index $\alpha=(\alpha^1,\dotsc, \alpha^n)$ with $|\alpha|=\alpha^1+\cdots+\alpha^n=l\geq0$,
\begin{align}
C_{l,\alpha}(x)\ce x^{\alpha}+\sum_{|\beta|=0}^{l-1}a_{l,\beta} x^{\beta},\quad x\in\D_n
\end{align}
with $a_{l,\beta}$ inductively defined by $a_{l,\beta}\ce \delta^\alpha_\beta$ for $\abs{\beta}=l$ and
\begin{align}
a_{l,\beta}\ce-\frac{\sum_{i=1}^n (\beta^i+2)(\beta^i+1)a_{l,(\beta^1,\dotsc,\beta^i+1,\dotsc,\beta^n)}}{(l-|\beta|)(l+|\beta|+2n+1)}\quad\text{for all $0\leq|\beta|\leq l-1$},
\end{align}
is an eigenfunction of $L_n$ in $\cl{\D}_n$ corresponding to the eigenvalue $\lambda_l^{(n)}=\frac{(n+l)(n+l+1)}{2}$.
\end{prop}

\section{Solution schemes for the \KFE}

Knowing the eigenfunctions (\cf Proposition~\ref{prop_Gegenb_n}), it is straightforward  to reconstruct the local  solution of \cite{kimura3}, \cite{BBMcK} (for details cf.~\cite{dat,THJ2}).
\begin{prop}\label{prop_sol_n}
For $n\in\N$ and any initial condition  $f\in L^2(\D_n)$, the \KFE corresponding to the diffusion approximation of the $n$-dimen\-sional \WF model   \eqref{eq_Ln} always possesses a unique solution $u\co{\big({\Delta}_{n}\big)}_{\infty}\map\R$ 
with $u(\,\cdot\,,t)\in C^\infty(\D_n)$ for each fixed $t\in(0,\infty)$ and $u(x,\,\cdot\, )\in C^\infty((0,\infty))$ for each fixed $x\in\Delta_{n}$. Furthermore, this solution (and all its spatial derivatives) may be extended continuously to the boundary~$\bd{\D}_n$.
\end{prop}

The regularity, which follows from the regularity of the generalised Gegenbauer polynomials, of course agrees  with standard PDE theory (\cf \eg~\cite{jost_pde}).

\subsection{Moments and the weak formulation of the \KFE}\label{sec_weak}

The  solution of equation~\eqref{eq_Ln} in $\D_n$ lacks conservation properties: As the smallest eigenvalue of $L_n$ is $\lambda_0^{(n)}=\frac{n(n+1)}2>0$, a solution tends to $0$  everywhere in $\D_n$ for $t\to\infty$. In particular, the total mass and all other moments are  not preserved. However,  these properties are an important property
 of the model, and what disappears in the interior of the simplex should accumulate in its boundary. After all, the process should continue after the loss of one or several alleles. We shall therefore introduce a suitable extended solution 
on the entire $\cl{\D}_n$. This solution will be derived from the conservation of the moments of the process. 

The moments of the $n$-dimen\-sional process as obtained as limits of those from  the underlying discrete model satisfy the 
\textit{moment evolution equations} 
\begin{align}\label{eq_moments_n}
\dd{}{t}\bar{\mu}_\alpha(t)=-\frac{\abs{\alpha}(\abs{\alpha}-1)}{2}\bar{\mu}_\alpha(t)+\sum_{i=1}^n\frac{\alpha_i(\alpha_i-1)}{2}\bar{\mu}_{\alpha-e_i}(t) 
\end{align}
for $\alpha=(\alpha_1,\dotsc,\alpha_n)$, $\abs{\alpha}\geq 1$, whereas $\dd{}{t}\bar{\mu}_0(t)=0$ (with $e_i$ denoting the multi-index $(0,\dotsc,0,1,0,\dotsc,0)$ with 1 appearing at the $i$\ord position). These moments can be defined as 
\begin{align}\label{eq_moments_U}
\bar{\mu}_\alpha(t)\ce\big[U,x^\alpha\big]_n\equiv\sum_{k=0}^n\:\intlim_{\bd_k\Delta_n} U(x,t)x^\alpha\,\leb_k(dx),\quad\text{$t\geq0$, $\alpha=(\alpha_1,\dotsc,\alpha_n)$},
\end{align}
with the hierarchical product  introduced in equation~\eqref{eq_prod_[]}. This now involves an integration over $\cl{\D}_n$, that is, including  the boundary $\bd\D_n$ of the state space, which corresponds to configurations of the model where some allele frequencies may be zero. Therefore, we introduce the capitalised $U\co(\cl{\D}_n)_\infty\map\R$  as an extended solution as the probability density function of the diffusion approximation of the $n$-dimensional  \WF process   on the entire $\cl{\D}_n$ (thus in particular $U\vert_{\D_n}$ is a solution of the \KFE\eqref{eq_Ln} in $\D_n$).

We shall now discuss the consistency  between the moments evolution equation~\eqref{eq_moments_n} and the Kolmogorov backward operator $L^*$ in $\cl{\D}_n$ as defined in equation~\eqref{eq_Ln*_def} (actually, the following considerations also hold for a generic product $[\dcd]$). Since $L^*$ has polynomial coefficients, it maps polynomials to polynomials, and we have
\begin{align}\notag
L^* x^\alpha
&= \half\sum_{i,j=1}^n\big(x^i(\delta^i_j-x^j)\big)\ddd{}{x^i}{x^j}x^\alpha\\\notag
&=\half\sum_{i=1}^n\alpha_i(\alpha_i-1)(x^{\alpha-e_i}-x^\alpha)-\half\sum_{i\neq j}\alpha_i\alpha_j x^\alpha\\
&=\half\sum_{i=1}^n\alpha_i(\alpha_i-1)x^{\alpha-e_i}-\half\abs{\alpha}(\abs{\alpha}-1)x^{\alpha}
\qquad\text{for $x\in\cl{\D}_n$},
\end{align}
which yields, using the notation of equation~\eqref{eq_moments_U}, 
\begin{align}
\big[U(t),L_n^*x^\alpha\big]_n=\half\sum_{i=1}^n\alpha_i(\alpha_i-1)\bar{\mu}_{\alpha-e_i}(t)-\half\abs{\alpha}(\abs{\alpha}-1)\bar{\mu}^{\alpha}(t)
\end{align}
where the right-hand side is equal to that of equation~\eqref{eq_moments_n}. Thus, if the moments equation is fulfilled for some probability density function~$U$, we may equivalently write
\begin{align}
\dd{}{t}\bar{\mu}_\alpha(t)=\bigg[\dd{}{t}{U(t)},x^\alpha\bigg]_n=\big[{U(t)},L_n^*x^\alpha\big]_n\quad \text{
for $t\in(0,\infty)$}.
\end{align}

Since the $x^\alpha$ generate the space of all polynomials and since the polynomials are dense in $C^\infty$, we therefore also have such relations for arbitrary  test functions $\phi$, 
\begin{align}\label{eq_weak_hier}
\bigg[\dd{}{t}{U(t)},\phi\bigg]_n=\big[{U(t)},L_n^*\phi\big]_n\quad \text{for $\phi\in C^\infty(\cl{\D}_n)$ and all $t\in(0,\infty)$}.
\end{align}
This is our  {weak formulation} \eqref{eq_weak_pre} of the \KFE\eqref{eq_Ln}. 
We may also write the  initial condition%
\footnote{Since we integrate over $\cl{\D}_n$, $f$ may now also be formulated as an extended initial condition on the entire $\cl{\D}_{n}$. Then, $f\vert_{\bd\D_n}\neq0$ would correspond to the process (partially) already starting on certain boundary instances. However, these parts of the process exactly evolve as a proper process of corresponding dimension, and hence do not yield any further insight into the nature of the process. For this reason, we will usually assume ${f}\vert_{\bd\D_n}\equiv 0$ or that~$f$ is extended that way if it is only given on~$\D_n$.} %
weakly as
\begin{align}
\big[U(\,\cdot\,,0),\phi\big]_n=\big[f,\phi\big]_n\quad\text{for all $\phi\in C^\infty(\cl{\D}_n)$,}
\end{align}
which requires no explicit regularity towards the boundary (but we shall need that its restriction to interior instances can be continuously extended to the corresponding boundary). 
Only, an integrability condition applies, which is  $U(\fdt,t),\dd{}{t}U(\fdt,t),f\in L^2\big(\bigcup_{k=0}^n\bd_k\D_n\big)$ for $t\geq0$.

Summarising our findings, we have:
\begin{lem}\label{lem_weak=moments+}
A function $U\co\big(\cl{\D}_n\big)_\infty\map\R$, $U(\fdt,t),\dd{}{t}U(\fdt,t)\in L^2\big(\bigcup_{k=0}^n\bd_k\D_n\big)$ for $t\geq0$ with corresponding moments $\bar{\mu}_\alpha(t)=[U(t),x^\alpha]_n$, $\alpha=(\alpha_1,\dotsc,\alpha_n)$, $t\geq0$ that satisfies the moments evolution equation~\eqref{eq_moments_n} also solves the weak formulation of the \KFE\eqref{eq_weak_hier} and conversely. 
\end{lem}

\subsection{A hierarchical extension of solutions and the boundary flux}\label{sec_hier_sol_fwd}

We shall now construct suitable boundary values as required for an extended solution $U\co\big(\cl{\D}_n\big)_\infty\map\R$. For that purpose,  we shall introduce  the concept of the boundary flux. The investigation of the boundary flux as the basis for a hierarchical solution 
scheme is the main new contribution of this paper and will follow below.

\begin{dfi}\label{dfi_allU}
For $\D_n^{(I_n)}$ with $I_n=\set{0,1,\dotsc,n}$ and a solution $u\colon\big(\D_n^{(I_n)}\big)_\infty\map\R$ of the \KFE\eqref{eq_Ln} for given $f\colon \D_n^{(I_n)}\map\R$ as in proposition~\ref{prop_sol_n},
a \itind{hierarchical extension}
\begin{align}
U &\co\big( \overline{\D_n^{(I_n)}}\big)_\infty\map\R\quad\text{with}\quad U(x,t)\ce \sum_{k=0}^n U_k(x,t)\chi_{\bd_k\D_n^{(I_n)}}(x)
\end{align}
is given by 
\begin{align}
U_{k}\colon\hspace*{-1pt}\big(\bd_k\D_n^{(I_n)}\big)_\infty\hspace*{-2.5pt}\map\R \hspace*{7.5pt}\text{with}\hspace*{7.5pt} U_{k}(x,t)\hspace*{-0.8pt}\ce\hspace*{-1.3pt}
\begin{cases}
u(x,t)		&\hspace*{-6pt}\text{for $x\in\D_n^{(I_n)}\hspace*{-1.5pt}\equiv\bd_n\D_n^{(I_n)}$}\\
U_{k,I_k}(x,t)  &\hspace*{-6pt}\text{for $x\in{\Delta}_{k}^{(I_k)}\hspace*{-1.5pt}\subset\bd_k\D_n^{(I_n)}\hspace*{-2pt},I_k\subset I_n$}\\
0		&\hspace*{-6pt}\text{else}
\end{cases}
\end{align}
for all $0\leq k\leq n$ 
and  
\begin{align}
&U_{k,I_k}\colon\big(\D_k^{(I_k)}\big)_\infty\map\R \quad\text{with}\quad U_{k,I_k}(x,t)\ce\intlim_0^t u_{k,I_{k}}^\tau(x,t-\tau)\,d\tau
\end{align}
for $0\leq k\leq n-1$ and for all subsets $I_k\subset I_n$. Here,  $u_{k,I_{k}}^\tau(x,t)\colon\big(\D_k^{(I_k)}\big)_\infty\map\R$ is a solution of 
\begin{equation}
\begin{cases}
L_k u(x,t) = \dd{}{t} u(x,t)			&\text{$(x,t)\in \big(\D_k^{(I_k)}\big)_\infty$}\\
u(x,0) = \sum_{I_{k+1}\supset 
I_k}G_{U_{k+1,I_{k+1}}}^\bot(x,\tau) &\text{$x \in \D_k^{(I_k)}$}\\
\end{cases}
\end{equation}
for all $\tau > 0$ as in proposition~\ref{prop_sol_n}
and  $G_{U_{k+1,I_{k+1}}}^\bot$ is the normal component of the flux of the {continuous extension} of $U_{k+1,I_{k+1}}$ to 
$\cl{\D_{k+1}^{(I_{k+1})}}$.
\end{dfi}

In general, the flux\sindex[not]{Guxt@$G_u(x,t)$} $G_u
\co(\Delta_n)_\infty\map\R^n$ of a solution $u\co(\Delta_n)_\infty\map\R^n$ of equation~\eqref{eq_Ln}
is given in terms of its components
\begin{align}
G_u^i(x,t)\ce-\half\sum_{j=1}^n\dd{}{x^j}(x^i(\delta^i_j-x^j)u(x,t))=-\half\sum_{j=1}^n\dd{}{x^j} (a^{ij}u(x,t)),\quad i=1,\dotsc,n.
\end{align}
In particular, we have
\begin{align}
\Div G_u=\sum_{i=1}^n\dd{}{x^i}G_u^i=-L_nu=- u_t.
\end{align} 
Again, this concept directly extends to boundary instances of $\cl{\D}_n$ if~$u$ extends to the boundary such that the extension is of class $C^2$ with respect to the spatial variables (which is the case for a solution from  proposition~\ref{prop_sol_n}).

With this  flux, we can now  state a generalised form of lemma~\ref{lem_adjoint}, which yields the adjointness  for the Kolmogorov operators $L_n$ and $L_n^\ast$ also for  non-vanish\-ing boundary terms:
\begin{prop}
\label{prop_mgf}
For $n\in\Np$ and $u,\phi\in C^2\big(\overline{\Delta}_{n}\big)$, we have
\begin{equation}\label{eq_mgf}
(L_n u,\phi)_{n}=-\intlim_{\bd_{n-1}\D_n}\!\phi\,G_u\cdot\nu \,d\leb_{n-1}+(u,L_n^*\phi)_{n} 
\end{equation}
where $G_u$ is the flux of~$u$ and $\nu$ the outward unit surface normal to $\bd\D_n$.
\end{prop}
\begin{proof}
We use the integration by parts formula
\begin{align}
\intlim_\Omega\dd{u}{x^i}\phi\,d\Omega=\intlim_{\bd\Omega}\phi u\,\nu^i\,d\bd\Omega-\intlim_\Omega u \dd{\phi}{x^i}\;d\Omega,
\end{align}
holding for a domain $\Omega$ with piecewise continuous boundary $\bd\Omega$, $u,\phi\in C^1(\overline{\Omega})$ and $\nu^i$ being the $i$\ord component of the outward unit surface normal to $\bd\Omega$. 
This yields
\begin{align}\notag
(L_n u,\phi)_{n} 
&=-\intlim_{\D_n}\sum_{i}\dd{}{x^i}G_u^i\phi\,d\leb_n\\
&=-\intlim_{\bd\D_n}\sum_{i}G_u^i\nu^i\phi\,d\leb_{n-1}+\intlim_{\D_n}\sum_{i}G_u^i\dd{}{x^i}\phi\,d\leb_n.
\end{align}
$\bigcup_{k=0}^{n-2}\bd_{k}\D_n$ clearly is a null set with respect to $\leb_{n-1}$, and we may hence replace the domain of integration of the first summand by $\bd_{n-1}\D_n$.
For the second term, we apply the integration by parts formula again (with the modified domain of integration):
\begin{multline}
\intlim_{\D_n}\sum_{i}G_u^i\dd{}{x^i}\phi\,d\leb_n
=-\intlim_{\bd_{n-1}\D_n}\half\sum_{i,j}x^i(\delta^i_j-x^j)u\nu^j\dd{}{x^i}\phi\,d\leb_{n-1}\\+\intlim_{\D_n}\half\sum_{i,j}a^{ij}u\ddd{}{x^i}{x^j}\phi\,d\leb_n.
\end{multline}
For the boundary integral over $\bd_{n-1}\D_n=\bigcup_{l=0}^n\D_{n-1}^{(I_n\smin\set{l})}$, we have $\nu^j=-\delta^j_l$ on $\D_{n-1}^{(I_n\smin\set{l})},l=1,\dotsc,n$ and $\nu^j=\frac1{\sqrt{n}}$ on $\D_{n-1}^{(I_n\smin\set{0})}$, which yields
\begin{align}
\sum_{j}x^i(\delta^i_j-x^j)u\nu^j
&=-x^i(\delta^i_l-x^l)u=0&&\text{on $\D_{n-1}^{(I_n\smin\set{l})}=\bigset{x^l=0}$}
\intertext{and}\notag
\label{eq_flux_face_0}
\sum_{j}x^i(\delta^i_j-x^j)u\nu^j
&=\frac1{\sqrt{n}}\sum_{j}x^i(\delta^i_j-x^j)u\\
&=\frac1{\sqrt{n}}x^i\Big(1-\sum_{j}x^j\Big)u=0&&\text{on $\D_{n-1}^{(I_n\smin\set{0})}=\Bigset{1-\sum_j x^j=0}$}.
\end{align}
Thus, the second integral over $\bd_{n-1}\D_n$ vanishes. Altogether, we have
\begin{align}
(L_n u,\phi)_{n} 
&=-\intlim_{\bd_{n-1}\D_n}\sum_{i}G_u^i\nu^i\phi\,d\leb_{n-1}+\intlim_{\D_n}u\half\sum_{i,j}a^{ij}\ddd{}{x^i}{x^j}\phi\,d\leb_n\\
&=-\intlim_{\bd_{n-1}\D_n}G_u\cdot\nu\phi\,d\leb_{n-1}+(u,L_n^*\phi)_{n}.\qedhere
\end{align}
\end{proof}
If $\phi\colon \cl{\D}_{n}\map\R$ is a polynomial of degree less than 2, we have  $L^*\phi=0$. Integrating the flux $G_u$ on $\bd_{n-1}\D_n$ over time as boundary values for a solution~$u$ of equation~\eqref{eq_Ln} (\resp for its continuous extension to $\bd\D_n$), proposition~\ref{prop_mgf}  already yields the behaviour for the \zeroth and the \first moment which is prescribed by the moments evolution equation~\eqref{eq_moments_U}. Thus, the total mass and the expectation value of the process are preserved in this case.

\label{pag_about_u_hat} This concept of a solution in $\D_n$ plus accumulated flux on the boundary $\bd_{n-1}\D_n$ is  not yet sufficient. In general it does not yield the desired evolution laws for moments of degree 2 or higher, nor does $\bd_{n-1}\D_n$ account for the full boundary $\bd\D_n$. This is resolved by assuming that the incoming flux rather evolves as if it were an $(n-1)$-dimen\-sional \WF process, \ie as a subsolution on $\bd_{n-1}\D_n$ instead of accumulating it on $\bd_{n-1}\D_n$ for $n\geq2$ statically. Iteratively repeating  the construction of boundary data to the boundary instances of  lower dimension by assessing the respective boundary flux of the subsolutions on each $\bd_{n-2}\D_{n-1}$ leads to Definition~\ref{dfi_allU}.

\begin{rmk}\label{rmk_hier_sol}
For a given solution~$u$ of equation~\eqref{eq_Ln}, the induced boundary functions $U_k$ on $\bd_k\D_n^{(I_n)}$ for $0\leq k\leq n-1$ of Definition~\ref{dfi_allU} in general do not satisfy the equation $\dd{}{t}U_k=L_kU_k$ in some $\D_k^{(I_k)}\subset\bd_k\D_n^{(I_n)}$. Consequently, they  are not solutions of the corresponding $k$-dimen\-sional problem~\eqref{eq_Ln} in $\D_k^{(I_k)}$.
\end{rmk}

\subsection{An application of the hierarchical conception}

For the hierarchically extended solution and the product $[\dcd]_n$, we may now continue the line of reasoning  of lemma~\ref{lem_adjoint} and proposition~\ref{prop_mgf}.
\begin{prop}\label{prop_allU}
A \txind{hierarchical extension} $U\colon\big(\overline{\D_n^{(I_n)}}\big)_\infty\map\R$ (\cf definition~\ref{dfi_allU}) of a solution~$u$ of the \KFE\eqref{eq_Ln} in $\D_n$ satisfies 
\begin{align}\label{eq_allU}
\bigg[\dd{}{t}U(t),\phi\bigg]_n
=\big [U(t),L^*\phi\big]_n
\end{align}
\text{for $\phi\in C^\infty\big(\overline{\D_n^{(I_n)}}\big)$ and for all $t\in (0,\infty)$}.
\end{prop}

\begin{proof}
By proposition~\ref{prop_mgf} we have for $U_n\equiv u$ and arbitrary $\phi\in C^\infty\big(\overline{\D_n^{(I_n)}}\big)$
\begin{align}
\bigg(\dd{}{t} U_n,\phi\bigg)_n=\big(L_n U_n,\phi\big)_n=
-\intlim_{\bd_{n-1}\D_n^{(I_n)}}\!\phi\,G_{U_n}^\bot \,d\leb_{n-1}+\big(U_n,L_n^*\phi\big)_n
\end{align}
where $G_{U_n}^\bot=G_{U_n}\cdot\nu$ denotes the (normal) flux of the {continuous extension} of $U_n$ to $\bd_{n-1}\D_n^{(I_n)}$.  
The  boundary integral can be expressed in terms of the evolution of the boundary function $U_{n-1}$ that lives on $\bd_{n-1}\D_n^{(I_n)}$. As this implies a hierarchical dependence on the particular subprocesses, we directly start our consideration for arbitrary $k\in\set{1,\dotsc,n}$. Then we have by proposition~\ref{prop_mgf} and for all $I_k\subset I_n$
\begin{align}
\int\limits_{\D_k^{(I_k)}} (L_k U_{k,I_k})\phi \,d\leb_{k}
=-\intlim_{\bd_{k-1}\D_k^{(I_k)}} \phi \,G^\bot_{U_{k,I_{k}}} \, d\leb_{k-1}+\int\limits_{\D_k^{(I_k)}} U_{k,I_k}L^*_k\phi \,d\leb_{k}
\end{align}
where $G_{U_{k,I_{k}}}$ again denotes the flux of the {continuous extension} of ${U_{k,I_{k}}}$ to $\bd_{k-1}\D_k^{(I_k)}$ (not to be confused with the proper boundary function $U_{k-1}$ on $\bd_{k-1}\D_n^{(I_n)}$).
Thus,  for the whole $k$-dimen\-sional boundary $\bd_k\D_n^{(I_n)}$ of $\D_n^{(I_n)}$, we have to  sum over all $\D_k^{(I_k)}\subset\bd_k\D_n^{(I_n)}$ resp.\ all subsets $I_k\subset I_n$. This yields (because of $\bigcup_{I_k\subset I_n} \D_k^{(I_k)}=\bd_k\D_n^{(I_n)}$ and the definition of~$U_k$)
\begin{align}
\int\limits_{\bd_k\D_n^{(I_n)}}( L_k U_{k}) \phi\,d\leb_{k}
&=\sum_{I_k\subset I_n}\intlim_{\bd_{k-1}\D_k^{(I_k)}}\phi\, G^\bot_{U_{k,I_{k}}} \, d\leb_{k-1}+\int\limits_{\bd_k\D_n^{(I_n)}}  U_{k} L_k^*\phi\,d\leb_{k}.
\end{align}
Transforming the boundary term using $\bigcup_{I_k\subset I_n} \bd_{k-1}\D_k^{(I_k)}=\bigcup_{I_{k-1}\subset I_n} \D_{k-1}^{(I_{k-1})}$ and employing the product notation, we get 
\begin{align}\label{eq_k-boundary}
\big(L_k U_{k}, \phi\big)_k
&=\sum_{I_{k-1}\subset I_n}\intlim_{\D_{k-1}^{(I_{k-1})}}\phi \sum_{I_k\supset{I_{k-1}}} G^\bot_{U_{k,I_{k}}} \, d\leb_{k-1}+ \big(U_{k}, L_k^*\phi\big)_k.
\end{align}
Now, the sum of fluxes appearing here may be expressed in terms of the evolution of the associated boundary function $U_{k-1,I_{k-1}}$ on $\D_{k-1}^{(I_{k-1})}$ for every $I_{k-1}\subset I_n$. By the chain rule, we have on $\D_{k-1}^{(I_{k-1})}$
\begin{align}\notag
\dd{}{t}U_{k-1,I_{k-1}}(x,t)
&=\dd{}{t}\intlim_{0}^t u_{k-1,I_{k-1}}^\tau(x,t-\tau)\,d\tau\\\notag
&=u_{k-1,I_{k-1}}^\tau(x,t-\tau)\big\vert_{\tau=t}+\intlim_{0}^t\dd{}{t} u_{k-1,I_{k-1}}^\tau(x,t-\tau)\,d\tau\\
&=u_{k-1,I_{k-1}}^{t}(x,0)+\intlim_{0}^t L_{k-1} u_{k-1,I_{k-1}}^\tau(x,t-\tau)
\end{align}
by the solution property of $u_{k-1,I_{k-1}}^\tau$. Interchanging $L_{k-1}$ with the $\tau$-integration and substituting $u_{k-1,I_{k-1}}^{t}(x,0)$ by the initial values as prescribed altogether yields
\begin{align}
-\sum_{I_k \supset I_{k-1}} G^\bot_{U_{k,I_{k}}}(x,t)=-\dd{}{t}U_{k-1,I_{k-1}}(x,t)+L_{k-1} U_{k-1,I_{k-1}}(x,t).
\end{align}
Multiplying this with $\phi$, integrating over $\D_{k-1}^{(I_{k-1})}$ and summing over all ${I_{k-1}}\subset{I_n}$ results in
\begin{align}\notag
 &-\sum_{{I_{k-1}\subset{I_n}}}\intlim_{\D_{k-1}^{(I_{k-1})}}\phi \sum_{I_k \supset I_{k-1}} G^\bot_{U_{k,I_{k}}} \, d\leb_{k-1}\\\notag
=&-\sum_{I_{k-1}\subset I_n}\intlim_{\D_{k-1}^{(I_{k-1})}}\phi\,\dd{}{t}U_{k-1,I_{k-1}}\, d\leb_{k-1}
+\sum_{I_{k-1}\subset I_n}\intlim_{\D_{k-1}^{(I_{k-1})}}\phi\, L_{k-1} U_{k-1,I_{k-1}}\, d\leb_{k-1}\\
=&-\bigg(\dd{}{t}U_{k-1},\phi\bigg)_{k-1}+\big(L_{k-1} U_{k-1},\phi\big)_{k-1}
\end{align}
because of $\bigcup_{I_{k-1}\subset I_n} \D_{k-1}^{(I_{k-1})}=\bd_{k-1}\D_n^{(I_n)}$. Combining this with  equation~\eqref{eq_k-boundary}, we get
\begin{align}
\big(L_k U_{k}, \phi\big)_k
&=-\bigg(\dd{}{t}U_{k-1},\phi\bigg)_{k-1}+\big(L_{k-1} U_{k-1},\phi\big)_{k-1}+ \big(U_{k}, L_k^*\phi\big)_k,
\end{align}
which -- by assumption -- holds for all $k\in\set{1,\dotsc,n}$. Hence, this formula may be iterated over~$k$, yielding
\begin{align}\notag
\bigg(\dd{}{t} U_{n}, \phi\bigg)_n
&=\big(L_n U_{n},\phi\big)_n\\\notag
\Leftrightarrow\quad
\bigg(\dd{}{t} U_{n}, \phi\bigg)_n+\bigg(\dd{}{t} U_{n-1}, \phi\bigg)_{n-1}
&=\big(U_{n},L_n^*\phi\big)_n + \big(L_{n-1} U_{n-1},\phi\big)_{n-1}\\\notag
&\hspace*{0.4em}\vdots\\
\Leftrightarrow\quad
\sum_{k=0}^n\bigg(\dd{}{t} U_{k}, \phi\bigg)_k
&=\sum_{k=1}^n \big(U_{k}, L_k^*\phi\big)_k+\big(L_{0} U_{0},\phi\big)_{0}.
\end{align}
The last summand on the right-hand side may (formally) be replaced by $\big(U_{0},L_{0}^*\phi\big)_{0}$ as they both vanish due to $L_0=L_0^*=0$, thus proving the assertion.
\end{proof}

By lemma~\ref{lem_weak=moments+} we immediately obtain: 
\begin{cor}
All moments $\bar{\mu}_\alpha(t)$, $t\geq0$ as defined in equation~\eqref{eq_moments_U} of a \txind{hierarchical extension} $U\colon\big(\overline{\D_n^{(I_n)}}\big)_\infty\map\R$ (\cf definition~\ref{dfi_allU}) of a solution~$u$ of the \KFE\eqref{eq_Ln} in $\D_n$ satisfy the moments evolution equation~\eqref{eq_moments_n}.
\end{cor}

\begin{proof}
For $\phi=1$ and $\phi=x^i$, we have $L^*(\phi)=0$, thus by equation~\eqref{eq_allU}
\[
\sum_{k=0}^n \bigg(\dd{}{t}U_k,\phi\bigg)_k=0.\qedhere
\]
\end{proof}

Thus, the \txind{hierarchical extension} of a solution of the \KFE\eqref{eq_Ln} via the flux of the solution yields the `right' boundary values on the entire $\bd\D_n$ in the sense that all moments of the process defined via the hierarchical product $[\dcd]_n$ in equation~\eqref{eq_moments_U} do behave like the limit of the moments underlying the discrete processes, which as well confirms the specific choice of $[\dcd]_n$.

Moreover, 
we may show that any extension of a solution of the \KFE\eqref{eq_Ln} to $\cl{\D}_n$ yielding the correct moments already coincides with the \txind{hierarchical extension} as in definition~\ref{dfi_allU}. This is due to lemma~\ref{lem_weak=moments+} and the following proposition:

\begin{prop}\label{prop_weak_unique} For any initial condition $f\in L^2(\D_n)$,  a solution $U\co{\big(\overline{\Delta}_{n}\big)}_{\infty}\map\R$  of the weak \KFE\eqref{eq_weak_hier} is uniquely defined on $\overline{\Delta}_{n}$.
\end{prop}

\begin{cor}
For any initial condition $f\in L^2(\D_n)$, 
a solution $U\co{\big(\overline{\Delta}_{n}\big)}_{\infty}\map\R$ of the weak \KFE\eqref{eq_weak_hier} 
coincides with the \txind{hierarchical extension} $U\co{\big(\overline{\Delta}_{n}\big)}_{\infty}\map\R$ (\cf definition~\ref{dfi_allU}) of a solution~$u$ of the (strong) \KFE\eqref{eq_Ln} in $\D_n$.
\end{cor}

For the proof of proposition~\ref{prop_weak_unique}, we need the following lemma:
\begin{lem}\label{lem_complete-weak}
The linear span of 
$\bigset{\omega_n\phi\in C_0^\infty\big (\cl{\Delta}_n\big)\big|\phi\text{ eigenfunction of $L_n$}}$ is dense in $C_c^\infty(\cl{\D}_n)$.
\end{lem}
\begin{proof}
From proposition~\ref{prop_Gegenb_n} we already see that the linear combinations of the eigenfunctions of $L_n$ are dense in $C^\infty(\cl{\D}_n)$. Dividing a function $f\in C_c^\infty(\cl{\D}_n)$ by $\omega_n$ (\cf lemma~\ref{lem_ef-shift}) again yields a function in $C_c^\infty(\cl{\D}_n)\subset C_0^\infty(\cl{\D}_n)$ as $\omega_n$ is in $C_0^\infty(\cl{\D}_n)$ itself and positive in the interior~${\Delta}_{n}$.
\end{proof}

\begin{proof}[Proof of proposition~\ref{prop_weak_unique}] 
Assume that $U^\prime\co{\big(\overline{\Delta}_{n}\big)}_{\infty}\map\R$ is another solution of equation~\eqref{eq_weak_hier} for a given initial condition~$f$.
We need to show that $U$ and $U^\prime$ agree on all $\bd_k\D_n\subset\overline{\Delta}_n$ for $k=n,\dotsc,0$. We start with $\bd_n\D_n\equiv\D_n$. 
For an eigenfunction $\phi\in C^\infty(\cl{\D}_n)$ of $L_n$ (corresponding to the eigenvalue $\lambda$), we obtain by lemma~\ref{lem_ef-shift} that $\psi\ce\omega_n\phi$ is an eigenfunction of $L_n^*$ with eigenvalue $\lambda$ and, by the properties  of $\omega_n$, that  $\psi\in C^\infty_0(\cl{\D}_n)$. For such a $\psi$, the weak \KFE\eqref{eq_weak_hier} then reduces to
\begin{alignat}{2}
\bigg(\dd{}{t}{U},\psi\bigg)_n	&= ({U},L_n^*\psi)_n	&&\equiv-\lambda(U,\psi)_n
\intertext{and}
\bigg(\dd{}{t}{U^\prime},\psi\bigg)_n	&=({U^\prime},L_n^*\psi)_n	&&\equiv-\lambda(U^\prime,\psi)_n
\end{alignat}
respectively. Consequently, by $t$-integration we have
\begin{align}\label{eq_agree}
 (U(t),\psi)_n	&=e^{-\lambda t}(U(0),\psi)_n,\\
(U^\prime(t),\psi)_n	&=e^{-\lambda t}(U^\prime(0),\psi)_n,
\end{align}
from which we obtain via  $U(0)=U^\prime(0)=f$ 
\begin{align}
(U(t),\psi)_n=(U^\prime(t),\psi)_n\quad\text{for $t\geq 0$}
\end{align}
and for all eigenfunctions $\psi$. Since the linear span of these functions is dense in $C_c^\infty(\cl{\D}_n)$ (\cf lemma~\ref{lem_complete-weak}), $U$ and $U^\prime$ agree in ${\Delta}_{n}$, indeed. 

Now, we proceed inductively. Assume that we have already shown that $U$ and $U^\prime$  agree on all $\bd_k\D_n\subset \cl{\D}_n$ with $k>m$. Then for an eigenfunction $\phi\co \cl{\Delta}_m\map\R$ of $L_m$ (corresponding to the eigenvalue $\lambda$), $\psi\ce\omega_m\phi$ again is an eigenfunction of $L_m^*$ with eigenvalue $\lambda$ and $\psi\in C^\infty_0(\cl{\D}_m)$. From any such $\psi\co\cl{\Delta}_m\map\R$, a function $\overline{\psi}\co \overline{\Delta^{(I_n)}_n}\map\R$ may be composed, %
\eg by copying $\psi$ to $\D_m^{(I_{m})}\subset\bd_m\D_n$ for all $I_m\subset I_n$ and employing convex combinations of the boundary values to spread to all higher dimensional (boundary) instances subsequently while putting $\overline{\psi}\ce 0$ on all lower dimensional boundary instances. Of course, $\overline{\psi}$ in general is  not an eigenfunction of $L^*$ in $\cl{\D}_n$, but we still have $(L^*\overline{\psi})\big\vert_{\D_m^{(I_{m})}}=L_m^* \psi=-\lambda\psi$ for all ${\D_m^{(I_{m})}}\subset\bd\D_n$.

For such a $\overline{\psi}$, the weak \KFE\eqref{eq_weak_hier} is converted into
\begin{align}
\bigg(\dd{}{t}{U},\overline{\psi}\bigg)_m
&= -\big({U},L_n^*\overline{\psi}\big)_m+\sum_{k=m+1}^n \bigg(\bigg(\dd{}{t}{U},\overline{\psi}\bigg)_k- \big({U},L_k^*\overline{\psi}\big)_k\bigg)
\intertext{and}
\bigg(\dd{}{t}{U^\prime},\overline{\psi}\bigg)_m
&= -\big({U^\prime},L_n^*\overline{\psi}\big)_m+\sum_{k=m+1}^n \bigg(\bigg(\dd{}{t}{U^\prime},\overline{\psi}\bigg)_k- \big({U^\prime},L_k^*\overline{\psi}\big)_k\bigg)
\end{align}
with the sums on the right agreeing as $U=U^\prime$ on all $\bd_k\D_n$ with $k > m$, hence
\begin{align}
\bigg(\dd{}{t}({U}-U^\prime),\overline{\psi}\bigg)_m= \big({U^\prime}-{U},L_n^*\overline{\psi}\big)_m \equiv \lambda \big ({U^\prime}-{U},\overline{\psi}\big)_m,
\end{align}
which yields -- analogously to our considerations above -- $U=U^\prime$ in $\bd_m\D_n$ on account of the completeness of the $\overline{\psi}$'s and the initial condition.
\end{proof}

Thus, with the additional assumptions that the moments of the process coincide with the limits of the moments of the underlying discrete processes, we altogether have:
\begin{thm}\label{thm_moments_n}
For $n\in\N$ and a given  initial condition $f\in L^2(\D_n)$, the \KFE corresponding to the diffusion approximation of the  $n$-dimen\-sional \WF model  \eqref{eq_Ln} always possesses a unique {extended solution} 
$U\co{\big(\overline{\Delta}_{n}\big)}_{\infty}\map\R$
in the sense that $U\vert_{\D_n}$ is a solution of equation~\eqref{eq_Ln} and  its moments
$\bar{\mu}_\alpha(t)\ce\big[U(t),x^\alpha\big]_n$, $t\geq 0$
(\cf equation~\eqref{eq_moments_U}) satisfy the $n$-dimen\-sional moments evolution  equation~\eqref{eq_moments_n}.
\end{thm}

\subsection*{Acknowledgements}
The research leading to these results has received funding from the European Research Council under the European Union's Seventh Framework Programme (FP7/2007-2013) / ERC grant agreement n$^\circ$~267087. J.\,H. and T.\,D.\,T. have also been supported by scholarships from the IMPRS ``Mathematics in the Sciences'' during earlier stages of this work. 

We also wish to thank the referee of \textit{Communications in Mathematical Sciences} for his/her useful suggestions which we think improve our presentation.

\end{document}